\documentclass{article}
\usepackage[utf8]{inputenc}
\usepackage{amssymb,amsthm,amsmath,amsfonts}
\usepackage{color, graphicx, enumerate}

\newtheorem{thm}{Theorem}
\newtheorem{lm}{Lemma}
\newtheorem{cor}{Corollary}
\newtheorem*{thm*}{Theorem}
\newtheorem*{lm*}{Lemma}
\newtheorem*{claim*}{Claim}

\theoremstyle{definition}
\newtheorem{remark}{Remark}

\numberwithin{equation}{section}
\numberwithin{figure}{section}

\newcommand{\R}{\mathbb{R}}
\newcommand{\e}{\varepsilon}
\newcommand{\mb}[1]{\mathbb{#1}}
\newcommand{\mc}[1]{\mathcal{#1}}
\newcommand{\Pb}[1]{\mathbb{P}\left( #1 \right)}
\newcommand{\p}[1]{\mathbb{P}\left( #1 \right)}
\newcommand{\E}{\mathbb{E}}
\newcommand{\dd}{\mathrm{d}}
\newcommand{\1}{\textbf{1}}
\newcommand{\modulo}{\ \textrm{mod} \;}
\newcommand{\fun}[3]{#1\colon #2 \longrightarrow #3}
\DeclareMathOperator{\Cov}{Cov}
\DeclareMathOperator{\id}{id}
\DeclareMathOperator{\sgn}{sgn}
\DeclareMathOperator{\CPE}{CPE}
\DeclareMathOperator{\CUE}{CUE}
\DeclareMathOperator{\GUE}{GUE}
\DeclareMathOperator{\GOE}{GOE}

\title{Tensor Products of Random Unitary Matrices}
\author{
Tomasz Tkocz$^1$\and 
Marek Smaczy\'nski$^2$ \and
Marek Ku\'s $^3$\and 
Ofer Zeitouni$^4$\and Karol \.Zyczkowski$^{2,3}$}

\date{March 16, 2012}
\begin{document}

\maketitle

\noindent 
$^1$Institute of Mathematics, University of Warsaw, Banacha 2, 02-097 Warsaw, Poland. \\ \texttt{tkocz@mimuw.edu.pl} \\ \\
$^2$The Marian Smoluchowski Institute of Physics, Jagiellonian University, Reymonta 4, 30-059 Cracow, Poland. \texttt{marek.smaczynski@uj.edu.pl} \\ \\
$^3$Center of Theoretical Physics, Polish Academy of Sciences, Al. Lotnikow 32/46, 02-668 Warsaw, Poland. \texttt{marek.kus@cft.edu.pl}, \texttt{karol@cft.edu.pl}\\ \\
$^4$School of Mathematics, University of Minnesota and Faculty of Mathematics, Weizmann Institute of Science, POB 26, Rehovot 76100, Israel. \texttt{zeitouni@math.umn.edu}

\begin{abstract}
Tensor products of $M$ random unitary matrices of size $N$
from the circular unitary ensemble are investigated.
We show that the spectral statistics of the tensor product
of random matrices becomes Poissonian if $M=2$, $N$ 
become large or $M$ become large and $N=2$.
%Apart of the standard level spacing distribution $P(s)$
%we analyze the statistics
%of extremal spacings $s_{\rm min}$  and $s_{\rm max}$.
%The average minimal spacing of a random unitary matrix of size $N$
%behaves as $N^{-1/(1+\beta)}$ for $\beta$ equal to $0,1$ and $2$
%for random unitary matrices of circular
%Poisson, orthogonal and unitary ensembles, respectively.
\end{abstract}



\section{Introduction}\label{sec:introduction}

Random matrices proved their usefulness in describing
the spectra of quantum systems, 
the classical analogues of which are chaotic \cite{H06,S99}. 
In particular, 
spectral properties of the
evolution operator of a deterministic quantum chaotic system 
seem to 
coincide with predictions obtained for the circular ensembles 
of random unitary matrices. The symmetry properties of the system 
determine which ensemble of random matrices is applicable.
Specifically,
if the physical system does not possess any time-reversal symmetry, one uses random unitary matrices of the circular unitary ensemble (CUE). 
%If such a symmetry exists symmetric unitary matrices of the circular orthogonal ensemble (COE)~\cite{Me04} are appropriate. Generally, for systems with time reversal symmetry the COE ensemble describes properly statistical properties of spectra, if we neglect additional subtleties caused by specific rotational symmetry features of systems with half-integer spin which are of no concern for our investigations reported in this paper. In the case of classically regular dynamics the spectrum of the evolution operator displays level clustering characteristic to the circular Poissonian ensemble (CPE) of diagonal random unitary matrices. To describe intermediate statistics one uses interpolating ensembles of unitary matrices \cite{PS91,LZ92,ZK96} or composed ensembles of unitary matrices \cite{Poz}. In the case of emerging chaos, in which the ch
% aotic layer covers only a fraction of the phase space of the classical system one may apply the distribution of Berry and Robnik, originally used for autonomous systems \cite{BR84}.

Statistical predictions obtained for ensembles of random matrices are also 
useful in analyzing generic properties of entangled states \cite{ZS01,HLW06}. 
In the theory of quantum information one deals with composite quantum system 
described in a Hilbert space with a tensor product structure. Thus any 
local unitary dynamics can be represented as a tensor product of 
unitary matrices, each describing time evolution of an individual subsystem.

Consider a quantum system consisting of $M$ subsystems. For simplicity we 
shall assume that each of them is described in $N$ dimensional Hilbert space, 
so that any local unitary dynamics can be written as 
$U=U_1 \otimes \dots \otimes U_M$, where $U_j\in U(N)$ for $j=1,\ldots,
M$. 
If the unitary dynamics of each subsystem is generic, 
the matrices U$_j$ can be represented by random matrices from the CUE.

The main aim of the present work is to analyze properties of the tensor product of random unitary matrices. We show that when either 
$N=2$ in the limit of a large number $M$ of subsystems,
or when $M=2$ in the limit of large subsystem size $N$, 
the point process obtained from the
spectrum of $U$, properly  rescaled,
becomes Poissonian, in the sense that its correlation
functions converge to that of a Poisson process. 
%Further,
%explicit formula for the spacing distributions 
%are derived in the simplest case of the two--qubit system.

%Apart from the standard level spacing distribution $P(s)$
%~\cite{Me04,Fo10} % Forrester 1991.
%we are going to analyze the distribution of the minimal spacing, $P(s_{\rm min})$. Observe that for any unitary matrix $U$ with spectrum $\exp(i\phi_j)$, the size of its minimal spacing $s_{\rm min}=\min_j(\phi_{j+1}-\phi_j)$, provides an information, to what extent the matrix $U$ is close to be degenerated. For completeness we are also going to study the size of the largest spacing  $s_{\rm max}$ defined similarly.
%
%The distribution of extremal eigenvalues of random Hermitian matrices was analyzed several years ago \cite{TA94}. Our current approach is in some sense analogous, as we investigate the extremal gaps between eigenvalues of a unitary matrix distributed along the unit circle. After our project was initiated we learned about a relevant paper of Arous and Bourgade \cite{AB10}, in which the distribution of extremal spacings was studied for random  matrices of circular unitary ensemble.

This paper is organized as follows. In section \ref{sec:spectral_statistics} we provide
some definitions and introduce our main results, Theorem
\ref{thm.cuecue} and \ref{thm.cue^}, and their corollaries; we also
provide numerical simulations that confirm the results. 
Section \ref{sec:spectral_statistics:two_nxn} provides the proof
of Theorem \ref{thm.cuecue} and of Corollary \ref{cor.cuecue},
while Section \ref{sec:spectral_statistics:M_2x2} is devoted to the proof
of Theorem \ref{thm.cue^} and of Corollary \ref{cor.cue^N}.

\section{Spectral statistics for tensor products of random unitary matrices}\label{sec:spectral_statistics}

The 
spectral statistics for two ensembles of unitary matrices will be the
focal
points of our investigation. 
The first case involves two unitary $N\times N$
matrices, whereas in the second we consider the tensor product of $M$
two-dimensional unitary matrices. As usual, we are interested in spectral
properties in the asymptotic limits of large matrices, i.e., respectively,
$N\to\infty$ and $M\to\infty$.

\subsection{Background and basic definitions}
We recall some standard
definitions and properties of some ensembles
of random unitary matrices.
The simplest situation is
a diagonal unitary matrix with eigenvalues being independently
drawn points on the unit circle. 
Such matrices form the 
\textbf{circular
Poisson ensemble}, $\mathbf{CPE}$ for short. 
The name reflects the fact that
for large matrices the number $n$ of eigenvalues inside an 
interval of the
length $L<<2\pi $ is approximately
Poisson-distributed
\begin{equation*}%\label{poisson}
p(L,n)\sim \frac{e^{-\lambda L}(\lambda L)^n}{n!}
\end{equation*}
with parameter $\lambda=N/2\pi$. 

Our main interest will be in
unitary matrices of size $N \times N$
drawn according to the Haar measure on the unitary group $U(N)$;
such a matrix is called a matrix from the
$\mathbf{CUE_N}$, where
$\CUE$ stands for \textbf{circular unitary
ensemble}.

Let $A_N$ be a $\CUE_N$ matrix. Denote by $(e^{i\theta^N_j})_{j=1}^N$ its
eigenvalues, where we assume that the eigenphases $\theta^N_j$ 
belong to the
interval $[0, 2\pi)$.  
The random vector $(\theta^N_1, \ldots, \theta^N_N)$
possesses
a density $P_{\CUE_N}$ with respect to the Lebesgue measure, which
%we refer to as
%a \textbf{joint probability
%distribution} of the eigenphases. For 
%the $\CUE_N$, this
was given by Dyson in his seminal paper \cite{Dys},
\begin{equation}\label{eq.cueNdensity}
	P_{\CUE_N}(\theta_1^N, \ldots, \theta_N^N) = C_N\prod_{1 \leq k < l \leq
N} |e^{i\theta^N_k} - e^{i\theta^N_l}|^2.
\end{equation}
This expression can be rewritten in the following form (see Paragraph 11.1 in
\cite{Me04})
\[
	P_{\CUE_N}(\theta_1^N, \ldots, \theta_N^N) = C_N(2\pi)^N\det \left
[S_N(\theta^N_k - \theta^N_l)\right ]_{k,l=1}^N,
\]
where
\begin{equation}\label{eq.kernel}
	S_N(x) = \frac{1}{2\pi}\frac{\sin\frac{Nx}{2}}{\sin\frac{x}{2}}.
\end{equation}
In particular
\[
	S_N(0) = \frac{N}{2\pi}.
\]

The set of eigenphases of a random unitary matrix can be seen as an
example of \textbf{a point process} $\chi_N$ on the interval $[0,2\pi)$
related to these eigenphases, by which we mean a random collection of points
$\{\theta^N_1, \ldots, \theta^N_N\}$ or, in other words, an integer-valued
random measure
\[
	\chi_N(D) = \sum_{k=1}^N \1_{\{\theta^N_k \in D\}}, \qquad D \subset
[0,2\pi),
\]
where $\1_X$ denotes the indicator function of $X$.

A possible way to describe a point process is to give its so-called
\textbf{joint intensities} or, as physicists usually say, \textbf{correlation
functions} $\fun{\rho^N_k}{(\R_+)^k}{\R_+}$, $k=1, 2, \ldots$. In our case
they might be defined simply as
\begin{equation}\label{eq.defint}
	\rho^N_k(x_1, \ldots, x_k) = \lim_{\e \to 0} \frac{1}{(2\e)^k}
\Pb{\exists j_1, \ldots, j_k \ \ |\theta^N_{j_s} - x_s| < \e, \ s = 1,
\ldots, k}.
\end{equation}
It is known \cite{Me04} that the process $\chi_N$ is
determinantal with joint intensities
\begin{equation}\label{eq.intensities}
	\rho^N_k(x_1, \ldots, x_k) = \det \left [S_N(x_s - x_t)\right
]_{s,t=1}^k.
\end{equation}
(Recall that a point process is called {\bf determinantal}
with kernel $K$ if its joint intensities can be written as
$\rho_k(x_1,\ldots,x_k)=\det[K(x_i,x_j)]_{i,j=1}^k$.) 
For $\CUE_N$ matrices,
due to the
translation invariance of the measures we have that
$K_N(x_i,x_j)=K_N(x_i-x_j)$, hence a
kernel is given by a function $K_N(x)$ of a single variable. We refer
to
\cite{AGZ} for more background on such determinantal processes.

By definition, the joint intensity $\rho^N_k$ equals 
$N!/(N-k)!$ times the $k$ dimensional marginal distribution of the vector
$(\theta^N_1, \ldots, \theta^N_N)$. Thus
\begin{equation}\label{eq.intensityintegral}
   \frac{(N-k)!}{N!}\int_{[0,2\pi)^k} \det \left [S_N(x_s - x_t)\right ]_{s,t=1}^k = 1.
\end{equation}

If we rescale properly the eigenphases of a $\CUE_N$ matrix it turns out that they
exhibit nice asymptotic behavior. Namely, it is clear that the point process
$\{\frac{N}{2\pi}\theta^N_1, \ldots, \frac{N}{2\pi}\theta^N_N\}$ is determinantal with the kernel $\frac{2\pi}{N}S_N\left(
\frac{x}{N/2\pi} \right)$.  Thanks to the fact that this function converges
when $N \to \infty$, we can give a precise analytic description of the limit
of the probability $\Pb{\frac{N}{2\pi}\theta^N_1\notin A, \ldots,
\frac{N}{2\pi}\theta^N_N \notin A}$, where $A \subset \R_+$ is a compact set
(see Theorem 3.1.1 in \cite{AGZ}).

In the case of $\CPE$ matrices the situation is even simpler. The
point process
related to the rescaled (by the factor $\frac{N}{2\pi}$) eigenphases of a
$\CPE_N$ matrix behaves for large $N$ as a Poisson point process with the
parameter $\lambda = 1$.

For point processes,
related to the correlation functions is
the notion of \textbf{level spacing distribution},
denoted by $P(s)$, which is defined 
for a point process
$\{\alpha\vartheta_1, \ldots, \alpha\vartheta_{N}\}$ of the properly rescaled
eigenphases $(\vartheta_j)_{j=1}^N$ of a random $N$-dimensional
unitary matrix by
\begin{equation}\label{eq.defp(s)}
   P(s) := \lim_{\e\to 0} \frac{1}{2\e} \frac{1}{N}\sum_{j=1}^N \Pb{s_j \in
   (s - \e, s + \e)},
\end{equation}
where
\begin{equation}\label{eq.defS_k}
   s_1 = \alpha\left( \vartheta_1' + 2\pi - \vartheta_{N}' \right), \qquad s_j =
   \alpha(\vartheta_j' - \vartheta_{j-1}'), \ 1 < j \leq N,
\end{equation}
and $(\vartheta_j')_{j=1}^{N}$ is the non-decreasing rearrangement of the
sequence $(\vartheta_j)_{j=1}^{N}$. 
The scaling factor $\alpha$ is chosen so that
the mean distance $\E s_j$ between two consecutive rescaled eigenphases is
$1$. In the cases of a $\CUE_N$ or $\CPE_N$ matrix, one
has $\alpha =
\frac{N}{2\pi}$. We should bear in mind that the level spacing distribution
of the Poisson point process with the parameter $\lambda=1$ is exponential with the
density
\begin{equation}\label{eq.expdistribution}
	P(s) = e^{-s}.
\end{equation}
Moreover, it is easy to check that
\[
	P_{\CPE_N}(s) \xrightarrow[N\to\infty]{} e^{-s}.
\]
Of course, the limit for the $\CUE_N$ is different.

\subsection{Statement of results}
We now present our main results for the two cases
under consideration.
\subsubsection{$M=2$, $N$ large}
We begin by considering 
two independent $\CUE$ matrices $A$ and $B$ of size $N$.
%with
%eigenphases $(\theta_j)_{j=1}^N$ and $(\phi_j)_{j=1}^N$ respectively. 
We are
interested in the asymptotic behavior of the 
%appropriately rescaled point
%process $\sigma_N$ of the 
eigenphases of the tensor product $A \otimes B$.
%The process is defined as follows
%\begin{equation}\label{eq.defsigma_N}
%   \sigma_N(D) := \sum_{k,l=1}^N \1_{\left\{\frac{N^2}{2\pi}\left( \theta_k +
%   \phi_l \mod 2\pi \right) \in D\right\}}, \qquad \textrm{for any compact
%   set $D \subset \R_+$}.
%\end{equation}
%We want to calculate the intensities $\rho^N_k$, $k = 1, 2, \ldots$ of the
%process $\sigma_N$ and investigate their convergence. 
Our first main
result is the following.
\begin{thm}\label{thm.cuecue}
%	Adopting the above notation we have for a fixed $k \geq 1$ that
Let $(\theta_j)_{j=1}^N$ and $(\phi_j)_{j=1}^N$ be the
eigenphases of two independent $\CUE_N$ matrices $A$ and $B$. 
Define the
point process $\sigma_N$ of rescaled eigenphases of the
matrix $A \otimes B$ as
\begin{equation}\label{eq.defsigma_N}
   \sigma_N(D) := \sum_{k,l=1}^N \1_{\left\{\frac{N^2}{2\pi}\left( \theta_k +
   \phi_l \mod 2\pi \right) \in D\right\}}, \qquad \textrm{for any compact
   set $D \subset \R_+$}.
\end{equation}
Let $\rho_k^N$, $k=1,2,\ldots$ be the intensities 
of the process $\sigma_N$. Then
\begin{equation}\label{eq.cuecueconv}
	\rho^N_k \xrightarrow[N\to\infty]{} 1,
\end{equation}
uniformly on any compact subset of $(\R_+)^k$.
%	
%	\begin{equation}\label{eq.cuecueconv}
%		\rho^N_k \xrightarrow[N\to\infty]{} 1,
%	\end{equation}
%	uniformly on any compact subset of $(\R_+)^k$.
\end{thm}
%Before we start the proof let us ponder some consequences of this
%convergence. Since all the joint intensities of the Poisson  point process
%with the parameter $\lambda=1$ are equal to unity, the 
Thus, Theorem \ref{thm.cuecue} relates
the statistical properties of a properly rescaled phase-spectrum of a large
$\CUE_N \otimes \CUE_N$ matrix to those of a
Poisson point process.
A (not immediate) corollary of the convergence of 
intensities is the following.
\begin{cor}\label{cor.cuecue}
   For the point process $\sigma_N$ defined in \eqref{eq.defsigma_N}, 
   %which
   %refers to the collection $\{\frac{N^2}{2\pi}\psi_1, \ldots,
   %\frac{N^2}{2\pi}\psi_{N^2}\}$ of the rescaled eigenphases
   %$(\psi_j)_{j=1}^{N^2}$ of a $\CUE_N \otimes \CUE_N$ matrix, we have
   \begin{equation}\label{eq.noeigencue}
   \begin{split}
      \Pb{\emph{\textrm{$\sigma_N$ has
      no rescaled eigenphase in the interval $[0, s]$}}} \\
      = \Pb{\sigma_N([0, s]) = 0} \xrightarrow[N\to\infty]{} e^{-s}, \qquad s > 0.
   \end{split}
   \end{equation}
   In particular
   \begin{equation}\label{eq.p(s)cuecue}
      P_{\CUE_N \otimes \CUE_N}(s)  \xrightarrow[N\to\infty]{} e^{-s},
   \end{equation}
   where the level spacing distribution $P_{\CUE_N \otimes \CUE_N}(s)$ is
   defined by \eqref{eq.defp(s)}
\end{cor}
\begin{figure}[htb!]
  \begin{center}
    \scalebox{1.0}{\includegraphics[width=1.0\textwidth]{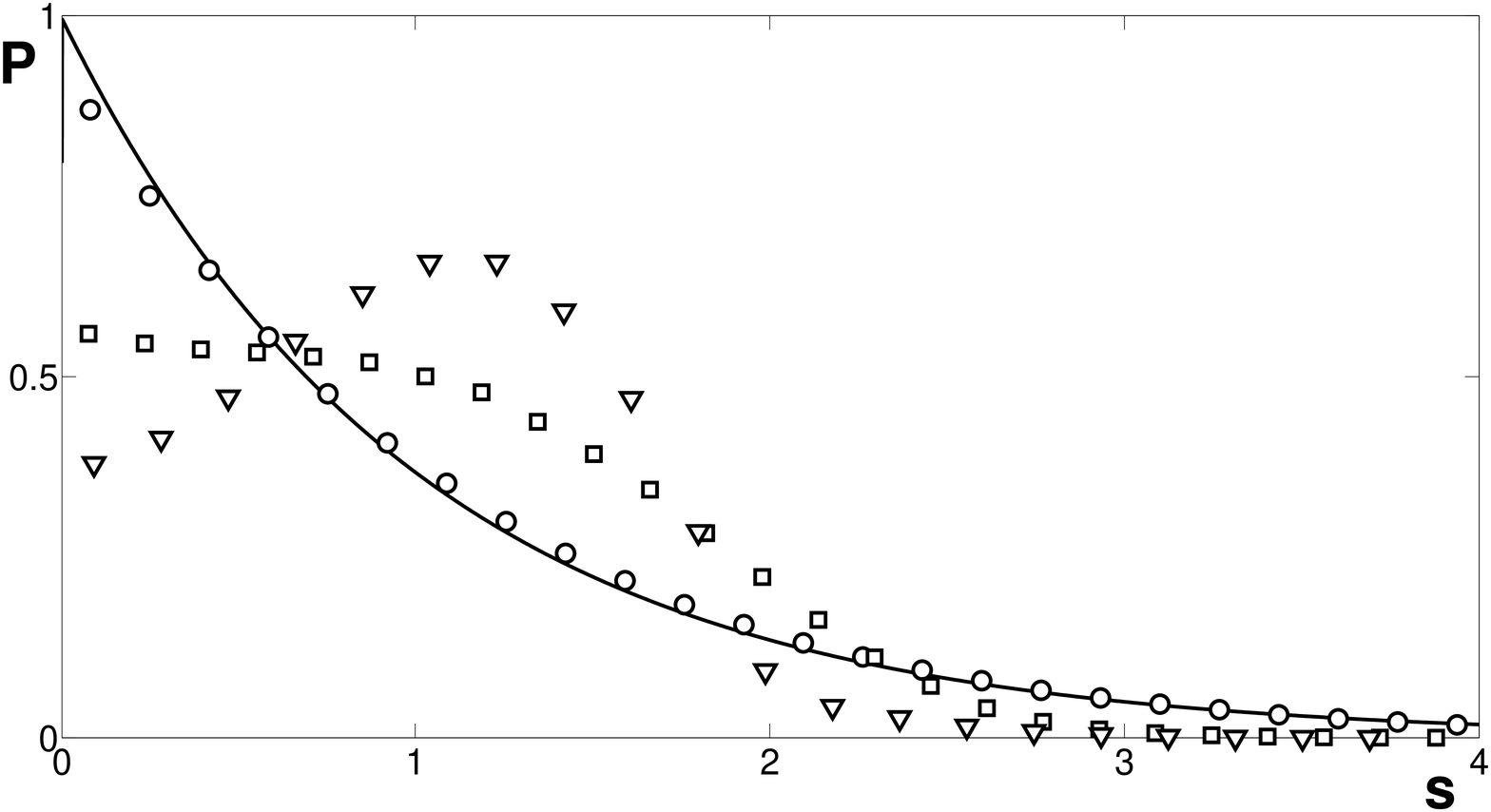}}
      \caption{The 
      level spacing distributions $P(s)$ for the tensor products of random unitary matrices  $CUE_N \otimes \CUE_N$ for $N = 2$ ($\triangledown$), $N = 3$ ($\square$), $N = 20$ ($\circ$). The symbols denote the numerical results respectively obtained for $2^{17}, 2^{16}, 2^{13}$ independent matrices, while the solid line represents the exponential distribution \eqref{eq.expdistribution}.}\label{fig.p(s)cuecue}
    \end{center}
\end{figure}
Our numerical results support \eqref{eq.cuecueconv}, i.e. the level spacing
distribution of the tensor product of two random unitary matrices of size N
is described asymptotically by the Poisson ensemble. 
The numerical data presented
in Figure \ref{fig.p(s)cuecue} reveals that just for $N = 20$ the difference
between $P_{\CUE_N \otimes \CUE_N}(s)$ and $P_{\CPE_N}(s)$ is negligible.

\subsubsection{$N=2$, $M$ large}
We next consider
$M$ independent $\CUE_2$
matrices $A_1, \ldots, A_M$ and study the asymptotic properties of
the phase-spectrum of a matrix $A_1 \otimes \ldots \otimes A_M$. 
Our main result is 
as
follows.
\begin{thm}\label{thm.cue^}
	Let $\theta_j^{1}, \theta_j^{2}$, $j = 1, \ldots, M$ be the
	eigenphases of independent $\CUE_2$ matrices $A_1, \ldots, A_M$. 
	Define the 
	point process $\tau_M$ of the
	rescaled eigenphases of a matrix $A_1 \otimes \ldots \otimes A_M$
	as
	\begin{equation}\label{eq.deftau_M}
		\tau_M(D) := \sum_{\epsilon = (\epsilon_1, \ldots, \epsilon_M) \in \{1, 2\}^M}
   \1_{\left\{\frac{2^M}{2\pi}\left( \theta^{\epsilon_1}_1 + \ldots +
   \theta^{\epsilon_M}_M \modulo 2\pi \right) \in D\right\}}, \qquad \textrm{for
   any compact set $D \subset \R_+$}.
	\end{equation}
	Then, for each $k$ there exists a continuous 
	function $\delta_k:\R_+\to \R_+$
	with $\delta_k(0)=0$
	so that
	for any mutually disjoint intervals $I_1, \ldots, I_k \subset \R_+$
	\begin{eqnarray*}
	  %\label{eq.thm2}
&&	\limsup_{M\to\infty} 
	\frac{\p{ \tau_M(I_1) > 0, \ldots, \tau_M(I_k) > 0 }}
	{|I_1|\cdot\ldots\cdot |I_k|} \leq ( 1+\delta_k(\max_j |I_j|))\,,\\
	&&\liminf_{M\to\infty} 
	\frac{\p{ \tau_M(I_1) > 0, \ldots, \tau_M(I_k) > 0 }}
	{|I_1|\cdot\ldots\cdot |I_k|} \geq ( 1-\delta_k(\max_j |I_j|))\,
	.
	\end{eqnarray*}
\end{thm}
Note that the statement of Theorem \ref{thm.cue^} is weaker than that of 
Theorem \ref{thm.cuecue}. This is due to the fact that stronger
correlations exist in the point process $\tau_M$, which prevent us from
discussing the convergence of its intensities to those of a Poisson
process. The mode of convergence is however strong enough to deduce 
interesting information, including the weak convergence of the processes.
We exhibit this by considering 
the behavior of the level 
spacings when $M$ tends to infinity.
\begin{cor}\label{cor.cue^N}
   For the point process $\tau_M$ defined in \eqref{eq.deftau_M}
   %, which
   %refers to the collection $\{\frac{2^M}{2\pi}\psi_1, \ldots,
   %\frac{2^M}{2\pi}\psi_{2^N}\}$ of the rescaled eigenphases
   %$(\psi_j)_{j=1}^{2^M}$ of a $\CUE_2^{\otimes M}$ matrix, 
   we have
   \begin{equation}\label{eq.noeigencue^}
   \begin{split}
      \Pb{\emph{\textrm{$\tau_M$ has
      no 
      eigenphase in the interval $[0, s]$}}} \\
      = \Pb{\tau_M([0, s]) = 0} \xrightarrow[M\to\infty]{} e^{-s}, \qquad s > 0.
   \end{split}
   \end{equation}
   In particular
   \begin{equation}\label{eq.p(s)cue^}
      P_{\CUE_2^{\otimes M}}(s)  \xrightarrow[M\to\infty]{} e^{-s},
   \end{equation}
   where the level spacing distribution $P_{\CUE_2^{\otimes M}}(s)$ is
   defined by \eqref{eq.defp(s)}.
\end{cor}
The relevant numerical results which confirm \eqref{eq.p(s)cue^} are shown in
Figure \ref{fig.p(s)cue^}. Again we may observe that it is enough to take
relatively small $M$ in order to get a good approximation of the spectrum of
a matrix $\CUE_2^{\otimes M}$ by the Poisson ensemble.
\begin{figure}[htb!]
  \begin{center}
    \scalebox{1.0}{\includegraphics[width=1.0\textwidth]{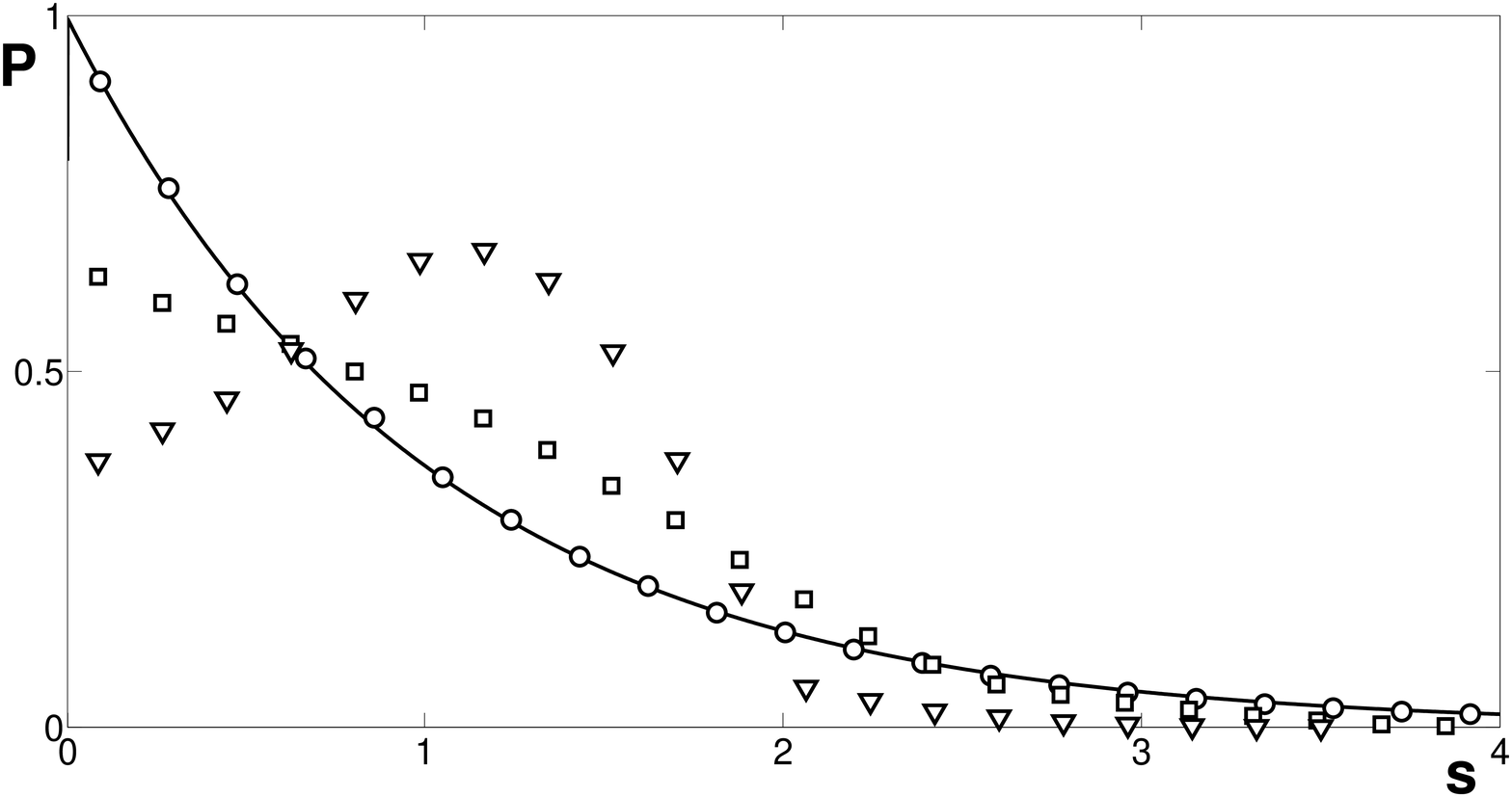}}
      \caption{Level spacing distributions $P(s)$ for the tensor products of
      random unitary matrices  $CUE_2^{\otimes M}$ for $M = 2$
      ($\triangledown$), $M = 3$ ($\square$), $M = 8$ ($\circ$). The symbols
      denote the numerical results respectively obtained for $2^{17}, 2^{16}, 2^{14}$ independent
      matrices, while the solid line represents the exponential distribution
      \eqref{eq.expdistribution}.}\label{fig.p(s)cue^}
    \end{center}
\end{figure}

\subsection{Discussion}
The convergence exhibited in Theorems \ref{thm.cuecue} and \ref{thm.cue^}, and
in their corollaries, 
is arguably not surprising: taking the tensor product introduces so many 
eigenphases ($N^2$ in the case of Theorem \ref{thm.cuecue},
$2^M$ in the case of Theorem \ref{thm.cue^}) that, after appropriate
scaling, the local correlations between adjacent eigenphases are not 
influenced by the long range correlation that is present due to the 
tensorization. One should however be careful in carying this heuristic 
too far: well known superposition and interpolation relations,
see \cite{FOR} and the discussion in \cite[Section 2.5.5]{AGZ},
show that the point process obtained by
the union of eigenvalues of, say, a $\GOE_N$ and $\GOE_{N+1}$ 
independent matrices, is closely related to that obtained from 
of a $\GUE_N$ matrix, and thus definitely not Poissonian.
This phenomenon had been also discussed in the physics literature
\cite{Jain}. Compared to that, the tensorization operation
appears to strongly decorrelate eigenphases on the local
level.

It is natural to try to
generalize Theorems \ref{thm.cuecue} and
\ref{thm.cue^} to other situations, where either $N$ or $M$ 
are finite but not necessarily equal to $2$, or both $N$ and $M$ 
go to infinity. While we expect similar methods to apply and yield
similar decorrelation results, there are several technical issues
to control, and we do not discuss such extensions here.

%In conclusion, the spectra of tensor products of two random unitary matrices
%of large size display Poisson level statistics. The same property holds for
%tensor product of a large number $M$ of random unitary matrices of size $2
%\times 2$.
%\begin{proof}
%????????????
%\end{proof}

\section{Tensor product of two $N \times N$ unitary matrices}
\label{sec:spectral_statistics:two_nxn}

We prove in this section Theorem \ref{thm.cuecue} and
Corollary \ref{cor.cuecue}, that correspond
to the case $M=2$ and $N$ large.
We start with an elementary observation.
Recall the kernel $S_N$, see \eqref{eq.kernel}.
\begin{lm} 
  \label{lem-oldclaim}
  For any $N \geq 1$
\begin{equation}\label{eq.SNbound}
	\sup_{x \in \R} |S_N(x)| = \frac{N}{2\pi}.
\end{equation}
\end{lm}
\begin{proof}
We show inductively that
\[
	|\sin(nu)| \leq n|\sin u|, \qquad \text{for} \ u \in \R, n \geq 1.
\]
Hence
\[
	|S_N(x)| = \frac{1}{2\pi}\left| \frac{\sin\left( N\frac{x}{2}
\right)}{\sin\frac{x}{2}}\right| \leq \frac{N}{2\pi}.
\]
For $x = 0$ we have equality, which completes the proof.
\end{proof}

\begin{proof}[Proof of Theorem \ref{thm.cuecue}]
We begin with setting $\tilde{x}_1, \ldots, \tilde{x}_k \geq 0$ and recalling
that by definition
\begin{align*}
	\rho^N_k(\tilde{x}_1, \ldots, \tilde{x}_k) = \lim_{\tilde{\e}\to 0}
\frac{1}{(2\tilde{\e})^k} \mb{P}\Bigg(& \exists \ \begin{matrix} (i) = (i_1,
\ldots, i_k) \in \{1, \ldots, N\}^k \\ (j) = (j_1, \ldots, j_k) \in \{1, \ldots, N\}^k \end{matrix}\ \ \forall s = 1, \ldots, k \\
& \frac{N^2}{2\pi}\left( \theta_{i_s} + \phi_{j_s} \mod 2\pi \right) \in
(\tilde{x}_s - \tilde{\e}, \tilde{x}_s + \tilde{\e}) \Bigg).
\end{align*}
Let us first of all get rid of the addition modulo $2\pi$ noticing that the
event, probability of which we want to compute, is the sum of $2^k$ mutually
exclusive events occurring when $\theta_{i_s} + \phi_{j_s}$ is in the
interval $[0, 2\pi)$ or $[2\pi, 4\pi)$. Thus we can write the sought after
probability as
\begin{equation}\label{eq:probtot}
	\sum_{(\eta) = (\eta_1,\ldots, \eta_k) \in \{0, 1\}^k} \Pb{\exists \begin{matrix}
(i) \\ (j) \end{matrix} \ \forall s \ \theta_{i_s} + \phi_{j_s} \in (\eta_s\cdot 2\pi +
x_s - \e, \eta_s\cdot 2\pi + x_s + \e)},
\end{equation}
where we denote $x_s = \frac{2\pi}{N^2}\tilde{x}_s$ and $\e =
\frac{2\pi}{N^2}\tilde{\e}$. Let us now concentrate solely on the first term
corresponding to the index $(\eta) = (\eta_1,\ldots, \eta_k) = (0,\ldots, 0)$ (the
other terms can be dealt with in the same manner). In order to take an
advantage of the independence we observe that the considered quantity
equals
\[
	\lim_{K\to\infty} \sum_{\substack{\ell_1,\ldots,\ell_k = 1, \\
2\pi\ell_s/K < x_s}}^K \Pb{ \exists \begin{matrix} (i) \\ (j) \end{matrix} \ \forall s
\ \begin{matrix} \theta_{i_s} \in (2\pi\ell_s/K - \pi/K, 2\pi\ell_s/K +
\pi/K) \\ \phi_{j_s} \in (x_s - 2\pi\ell_s/K - \e, x_s - 2\pi\ell_s/K + \e)
\end{matrix}},
\]
where the constrains $2\pi\ell_s/K < x_s$ are the result of the fact that $\theta_{i_s} + \phi_{j_s} \in (0 \cdot 2\pi + x_s - \e, 0 \cdot 2\pi + x_s + \e)$, for $(\eta) = 0$, so, in particular, that $\theta_{i_s} < x_s + \e$.
Exploiting the independence we obtain that the last expression equals
\begin{equation}\label{eq.justafterxploitingindependence}
   \begin{split}
   	\sum_{\substack{\ell_1,\ldots, \ell_k = 1 \\ 2\pi\ell_s/K < x_s}}^K & \Pb{ \exists (i) \forall s \ \ \theta_{i_s} \in (2\pi\ell_s/K - \pi/K, 2\pi\ell_s/K + \pi/K)} \\
	& \ \cdot \Pb{ \exists (j) \forall s \ \ \phi_{j_s} \in (x_s - 2\pi\ell_s/K - \e, x_s - 2\pi\ell_s/K + \e)}.
   \end{split}
\end{equation}
Now observe that for a determinantal point process $\{\alpha_j\}_{j=1}^N$ with a kernel $K$ and fixed numbers $u_1, \ldots, u_k$ we have
\begin{equation}\label{eq.keylemmaofrepeatedindices}
\begin{split}
   & \Pb{\exists (i) \in \{1, \ldots, N\}^k \ \forall s = 1, \ldots, k \ \ \alpha_{i_s} \in (u_s - \delta, u_s + \delta)}  \\
   &= \sum_{p=1}^k \sum_{\pi \in \mathfrak{S}(k, p)} \lambda_\pi (u_1, \ldots, u_k) \left( (2\delta)^p \det\left[ K(u_{\pi(s,1)}, u_{\pi(t,1)}) \right]_{s,t=1}^p + o(\delta^p)\right),
\end{split}
\end{equation}
where $\mathfrak{S}(k, p)$ is the collection of all partitions into $p$ non-empty pairwise disjoint subsets of the set $\{1, \ldots, k\}$. By this we mean that if $\pi$ is such a partition then
\[
	\pi = \{\{\pi(1,1),\ldots, \pi(1,\sharp \pi(1))\}, \ldots, \{\pi(p,1),\ldots, \pi(p,\sharp \pi(p))\}\},
\]
where $\sharp \pi(q)$ is cardinality of the $q$-th block of the partition $\pi$. Moreover, to compactify
the notation, we attach to a partition $\pi$ the function $\fun{\lambda_\pi}{\R^k}{\{0,1\}}$, defined as
\[
	\lambda_\pi(u_1, \ldots, u_k) = \1_{\{u_{\pi(1,1)} = \ldots = u_{\pi(1,\sharp \pi(1))}, \ldots, u_{\pi(p,1)} = \ldots = u_{\pi(p,\sharp \pi(p))}\}}(u_1, \ldots, u_k).
\]
Applying this to formula \eqref{eq.justafterxploitingindependence} we obtain
\begin{align*}
	\sum_{\substack{\ell_1,\ldots, \ell_k = 1 \\ 2\pi\ell_s/K < x_s}}^K \sum_{p_1, p_2 = 1}^k \sum_{\substack{\pi_1 \in \mathfrak{S}(k, p_1) \\ \pi_2 \in \mathfrak{S}(k, p_2)}} 
	& \lambda_{\pi_1}\left( (2\pi \ell_s/K)_{s=1}^k \right)\lambda_{\pi_2}\left( (x_s - 2\pi\ell_s/K)_{s=1}^k \right) \\ 
	& 
	\!\!\!\!\!\!
	\!\!\!\!\!\!
	\!\!\!\!\!\!
	\!\!\!\!\!\!
	\!\!\!\!\!\!
	\!\!\!\!\!\!
	\!\!\!\!\!\!
	\!\!\!\!\!\!
	\!\!\!\!\!\!
	\cdot\left( \left( \frac{2\pi}{K} \right)^{p_1} \det\left[ S_N(2\pi\ell_{\pi_1(s,1)}/K - 2\pi\ell_{\pi_1(t,1)}/K) \right]_{s,t=1}^{p_1} + o(1/K^{p_1}) \right) \\
	& 
	\!\!\!\!\!\!
	\!\!\!\!\!\!
	\!\!\!\!\!\!
	\!\!\!\!\!\!
	\!\!\!\!\!\!
	\!\!\!\!\!\!
	\!\!\!\!\!\!
	\!\!\!\!\!\!
	\!\!\!\!\!\!
	\cdot\left( \left( 2\e \right)^{p_2} \det\left[ S_N\left(x_{\pi_2(s,1)} - 2\pi\ell_{\pi_2(s,1)}/K - x_{\pi_2(t,1)} + 2\pi\ell_{\pi_2(t,1)}/K\right) \right]_{s,t=1}^k + o(\e^{p_2}) \right)\,.
\end{align*}
Performing the limit $K \to \infty$ we notice that only the terms corresponding to $p_2 = k$ do not vanish, for,  otherwise, $\lambda_{\pi_2}$ would give nontrivial relations for $(\ell)$ which altogether with $\lambda_{\pi_1}$ make the sum over $(\ell)$ of at most $O(K^{p_1-1})$ terms. Recall that $\e/\tilde{\epsilon} = 2\pi/N^2$. 
Thus, taking the limit $\tilde{\e} \to 0$,
the extra factor $(2\pi/N^2)^k$ is produced, so we finally find that the considered term contributes
\begin{align*}
	\sum_{p = 1}^k \frac{1}{N^{k-p}}\sum_{\pi \in \mathfrak{S}(k, p)} \frac{1}{(2\pi)^p}\int_{\substack{[0,2\pi)^k \\ y_s <
x_s}} & \lambda_{\pi}(y_1, \ldots, y_k) \det\left[ \frac{2\pi}{N}S_N(y_{\pi(s,1)} - y_{\pi(t,1)}) \right]_{s,t=1}^p \\
&\cdot \det\left[ \frac{2\pi}{N}S_N(x_s - y_s -
x_t + y_t) \right]_{s,t=1}^k \dd \mathcal{H}_p(y_1,\ldots, y_k)
\end{align*}
to $\rho^N_k(x_1, \ldots, x_k)$, where $\mathcal{H}_p$ denotes the $p$-dimensional Hausdorff measure in $\R^k$. As already mentioned the other terms in
(\ref{eq:probtot}) can be calculated in a similar way, only the limits of the
integration have to be changed. Summing up, we get
\begin{equation}\label{eq.intensityprefinal}
\begin{split}
	\rho^N_k(x_1, \ldots, x_k) = & \sum_{(\eta) \in \{0, 1\}^k}\sum_{p = 1}^k \frac{1}{N^{k-p}}\sum_{\pi \in \mathfrak{S}(k, p)}
\frac{1}{(2\pi)^p}\int_{\substack{A_{(\eta)}}} \Bigg( \lambda_{\pi}(y_1, \ldots, y_k) \\
&\cdot \det\left[ \frac{2\pi}{N}S_N(y_{\pi(s,1)} - y_{\pi(t,1)}) \right]_{s,t=1}^p \\
&\cdot \det\left[ \frac{2\pi}{N}S_N(2\pi\eta_s + x_s - y_s - 2\pi\eta_t - x_t + y_t) \right]_{s,t=1}^k \Bigg) \dd \mathcal{H}_p(y_1,\ldots, y_k),
\end{split}
\end{equation}
where the subset $A_{(\eta)}$ of $[0, 2\pi)^k$ is the set of all $(y_1,
\ldots, y_k)$ such that either $y_s < x_s$ if $\eta_s = 0$, or $y_s \geq x_s$
if $\eta_s = 1$ for $s = 1, \ldots, k$.

To proceed we have to investigate the asymptotic behavior of the integrand
in (\ref{eq.intensityprefinal}). We will do it again only for $(\eta) = (0,
\ldots, 0)$, observing that an adaptation to other terms is straightforward. We start with the term $p = k$.
Then the integrand is a product of two determinants of matrices of size $k$, so applying to each of them the permutation definition and extracting the term referring to the trivial permutations, we find it equals
\begin{equation}
  \label{eq-oof}
	\left(\frac{2\pi}{N}S_N(0)\right)^{2k} + \sum_{\substack{\sigma \neq \id \, \textrm{or} \, \tau \neq
\id}} \sgn \sigma \sgn \tau\prod_{i=1}^k \frac{2\pi}{N}S_N(y_i - y_{\sigma(i)})
\prod_{j=1}^k \frac{2\pi}{N}S_N(x_j - y_j - x_{\tau(j)} + y_{\tau(j)}),
\end{equation}
where the summation involves all permutations $\sigma$ and $\tau$ of $k$
indices. The first term\\
 $(2\pi S_N(0)/N)^{2k}=1$, after substituting in \eqref{eq.intensityprefinal},
gives simply
\[
	\frac{1}{(2\pi)^k}\sum_{(\eta) \in \{0, 1\}^k} \int_{\substack{A_{(\eta)}}} \left(\frac{2\pi}{N}S_N(0)\right)^{2k} = 1.
\]
We will show that the second term in \eqref{eq-oof} after being put into \eqref{eq.intensityprefinal} vanishes in the limit. We consider here only the case
$k=2$ to explain the main idea. The terms involving more factors can be treated
along the same lines. The sum over $\sigma$ and $\tau$ reduces to
\begin{equation}\label{eq.lastterm}
\begin{split}
	-\left(\frac{2\pi}{N}S_N(0)\right)^2\left( \left(\frac{2\pi}{N}S_N(y_1 - y_2)\right)^2 + \left(\frac{2\pi}{N}S_N(x_1 - y_1 - x_2 + y_2)\right)^2 \right) \\
	+
\left(\frac{2\pi}{N}S_N(y_1 - y_2)\right)^2\left(\frac{2\pi}{N}S_N(x_1 - y_1 - x_2 + y_2)\right)^2.
\end{split}
\end{equation}
Let us for instance deal with the last term in equation \eqref{eq.lastterm}. Putting it into
\eqref{eq.intensityprefinal} we arrive at
\[
	\frac{1}{(2\pi)^2}\sum_{(\eta)}\int_{A_{(\eta)}} \left( \frac{2\pi}{N}S_N(y_1 - y_2)
\right)^2\left( \frac{2\pi}{N}S_N(x_1 - y_1 - x_2 + y_2) \right)^2.
\]
Taking a quick look at the integrand we see that the above expression goes to
$0$ when $N \to \infty$ by Lebesgue's dominated convergence theorem, for
$\frac{1}{N}S_N(u) \xrightarrow{N\to\infty} 0$, when $u \neq 0$, and the
appropriate bound \eqref{eq.SNbound} follows from
Lemma \ref{lem-oldclaim}.
%is provided with the claim (see the
%proof of Corollary \ref{cor.cuecue}).

For the terms corresponding to $k < p$, we easily notice that thanks to the factor $\frac{1}{N^{k-p}}$ they converge to $0$. The proof is now complete.
\end{proof}

\begin{remark}
  \label{rem-1}
By virtue of formula \eqref{eq.intensityprefinal} the joint intensities $\rho^N_k$ can be estimated
as
\begin{align*}
	\sup_{\R^k} |\rho_k^N| \leq & \frac{1}{N^k} \sup_{u_1, \ldots, u_k \in \R} \det\left[ \frac{2\pi}{N}S_N(u_s - u_t) \right]_{s,t=1}^k  \\
	&\cdot \sum_{p=1}^k \sharp \mathfrak{S}(k, p) \int_{[0,2\pi)^p} \det\left[ S_N(y_s - y_t) \right]_{s,t=1}^p \dd y_1 \ldots \dd y_p,
\end{align*}
where $\sharp X$ denotes cardinality of a set $X$.
Using Hadamard's inequality (see, e.g. (3.4.6) in \cite{AGZ}) for the first
term, the observation \eqref{eq.intensityintegral} for the second one, and
finally \eqref{eq.SNbound} we obtain
\[
\begin{split}
   \sup_{\R^k} |\rho_k^N| &\leq \frac{1}{N^k} \left( \sup \left|\frac{2\pi}{N}S_N\right| \right)^k k^{k/2} \sum_{p=1}^k \sharp \mathfrak{S}(k, p) \frac{N!}{(N - k)!} = k^{k/2}\frac{1}{N^k} \sum_{p=1}^k \sharp \mathfrak{S}(k, p) \frac{N!}{(N - k)!}.
\end{split}
\]
Due to the  well-known combinatorial fact that
\[
	\sum_{p=1}^k \sharp \mathfrak{S}(k, p) x(x-1)\cdot\ldots(x-p+1) = x^k,
\]
($\sharp \mathfrak{S}(k, p)$ is the Stirling number of the second kind, consult e.g. \cite{GKP}) we may conclude with an useful bound
\begin{equation}\label{eq.intensitiesbound}
   \sup_{\R^k} |\rho_k^N| \leq k^{k/2}.\qed
\end{equation}
\end{remark}

\begin{proof}[Proof of Corollary \ref{cor.cuecue}]

For the proof of \eqref{eq.noeigencue} we have to calculate the probability
of the event that there is no rescaled eigenphase in a given interval. This
is done in the following lemma.
\begin{lm}\label{lm.norescaled}
   Let $\chi$ be a point process related to the eigenphases, possibly
   rescaled, of a $CUE_N$ matrix $A_N$ with the joint intensities $\rho_k$,
   $k=1,2,\ldots$ (so $\rho_\ell \equiv 0$, for $\ell > N$). Then for any
   compact set $D$
   \begin{equation}\label{eq.noeigenphaseininterval}
      \Pb{\chi(D) = 0} = 1 + \sum_{\ell = 1}^\infty \frac{(-1)^\ell}{\ell !}
      \int_{D^\ell} \rho_\ell.
   \end{equation}
\end{lm}
\begin{proof}
Clearly, we have
\[
  	\Pb{\chi(D) = 0} = 1 - \sum_{k = 1}^N \Pb{\chi(D) = k}.
\]
One way to compute the probability $\Pb{\chi(D) = k}$ is to use the notion of
J\'anossy densities $j_{D,k}(x_1,\ldots,x_k)$ (see Definition 4.2.7 in
\cite{AGZ}). They can be expressed in terms of the joint intensities as
\begin{equation}\label{janossy1}
j_{D,k}(x_1,\ldots,x_k)=\sum_{r=0}^\infty\frac{1}{r!}
(-1)^r\rho_{k+r}(x_1,\ldots,x_k,\underbrace{D, \ldots, D}_r),
\end{equation}
where
\begin{equation}\label{janossy2}
\rho_{k+r}(x_1,\ldots,x_k,\underbrace{D, \ldots, D}_r)=
\int_{D^r}\rho_{k+r}(x_1,\ldots,x_k,y_1,\ldots, y_r)dy_1\cdots dy_r.
\end{equation}
They exist whenever
\begin{equation}\label{janossy3}
\sum_k\int_{D^k}\frac{k^r\rho_k(x_1,\ldots,x_k)}{k!}dx_1\cdots dx_k<\infty,
\end{equation}
which is clearly fulfilled in our case, as $\rho_\ell \equiv 0$ for $\ell
> N$. Moreover, the vanishing of $\rho_\ell$ for large enough $\ell$
makes every sum in the following finite so we will not have troubles with
interchanging the order of summations.

In terms of the J\'anossy intensities, the probability $\Pb{\chi(D) = k}$
reads as (see Equation (4.2.7) of \cite{AGZ})
\begin{equation}\label{janossy4}
\Pb{\chi(D) = k}=\frac{1}{k!}\int_{D^k}j_{D,k}(x_1,\ldots,x_k)dx_1\cdots dx_k,
\end{equation}
and, consequently,
\begin{align*}
   \Pb{\chi(D) = 0} &= 1 - \sum_{k = 1}^N \frac{1}{k!} \int_{D^k} j_{D,k} \\
   &= 1- \sum_{k=1}^n \frac{1}{k!} \int_{D^k} \sum_{r \geq 0} \frac{(-1)^r}{r!} \rho_{k+r}(x_1, \ldots, x_k, \underbrace{D, \ldots, D}_r) \dd x_1 \ldots \dd x_k \\
   &= 1 - \sum_{k \geq 1}\sum_{r \geq 0} \frac{1}{k!}\frac{(-1)^r}{r!} \int_{D^{k+r}}\rho_{k+r} = 1 - \sum_{k \geq 1}\sum_{\ell \geq k} \frac{1}{k!}\frac{(-1)^{\ell - k}}{(\ell - k)!} \int_{D^\ell}\rho_\ell \\
   &= 1 - \sum_{\ell \geq 1} \Bigg[ \sum_{k \geq 1} {\ell \choose k} (-1)^k \Bigg] \frac{(-1)^\ell}{\ell !} \int_{D^\ell}\rho_\ell = 1 + \sum_{\ell \geq 1} \frac{(-1)^\ell}{\ell !} \int_{D^\ell} \rho_\ell.
\end{align*}
\end{proof}

Lemma \ref{lm.norescaled} applied to the process $\sigma_N$ yields
\[
	\Pb{\sigma_N([0,s])=0} = 1 + \sum_{\ell \geq 1} \frac{(-1)^\ell}{\ell !} \int_{[0,s]^\ell} \rho_\ell^N.
\]
To pass to the limit $N \to \infty$ we need an appropriate bound on the
intensities $\rho_\ell^N$. In Remark \ref{rem-1}
%In the course of (momentarily postponed) proof of Theorem \ref{thm.cuecue} 
we showed 
that $|\rho_\ell^N| \leq \ell^{\ell/2}$ (see \eqref{eq.intensitiesbound}). 
Therefore, by Lebesgue's dominated convergence theorem, we get
\[
	\lim_{N\to\infty} \Pb{\sigma_N([0,s])=0} = 1 + \sum_{\ell \geq 1}
\frac{(-1)^\ell}{\ell !} \int_{[0,s]^\ell} \lim_{N\to\infty} \rho_\ell^N = 1
+ \sum_{\ell \geq 1} \frac{(-1)^\ell}{\ell !} s^\ell = e^{-s}.
\]
This completes the proof of \eqref{eq.noeigencue}.

The formula \eqref{eq.p(s)cuecue} follows now from a relation connecting the
probability $E(0,s)$ that a randomly chosen interval of length $s$ is
free from eigenphases with the level spacing distribution $P(s)$,
\eqref{eq.defp(s)} (see equation (6.1.16a) in \cite{Me04}),
\begin{equation}\label{}
P(s)=\frac{\dd^2}{\dd s^2} E(0;s).
\end{equation}
We have just showed that $ \lim_{N\to\infty}\Pb{\sigma_N([0,s])=0} = E(0;s) =
e^{-s}$. Thus, indeed
\[
	\lim_{N\to\infty}P_{\CUE_N \otimes \CUE_N}(s) = \frac{\dd^2}{\dd s^2} e^{-s} = e^{-s}.
\]
\end{proof}

\section{Tensor product of $M$ unitary matrices of size $2 \times 2$}
\label{sec:spectral_statistics:M_2x2}

Now we will prove Theorem \ref{thm.cue^}. In the course of the proof we will need three lemmas. Let us start with them.
\begin{lm}\label{lm.q.g.sum}
	Fix a positive integer $s$ and a number $\gamma \in (0,1/s)$. For each positive integer $n$ let us define the set
$
	\mc{L}_n = \{\ell = (\ell_1, \ldots, \ell_s) \ | \ \mb{Z} \ni \ell_j \geq 0, \sum_{j=1}^s \ell_j = n\}.
$ 
	Then
	\begin{equation}\label{eq.lm.q.g.sum}
		\sum_{\ell \in \mc{L}_n, \exists j \; \ell_j/n \leq \gamma} \frac{1}{s^n}\frac{n!}{\ell!} = 1 - \sum_{\ell \in \mc{L}_n, \forall j \; \ell_j/n > \gamma} \frac{1}{s^n}\frac{n!}{\ell!} \xrightarrow[n\to\infty]{} 0.
	\end{equation}
\end{lm}
\noindent Here, we 
adopt the convention that $\ell! = \ell_1!\cdot\ldots\cdot\ell_s!$.
\begin{proof}
First observe  that
\begin{align*}
	\sum_{\ell , \exists j \; \ell_j/n \leq \gamma} \frac{1}{s^n}\frac{n!}{\ell!} &\leq s\sum_{\ell , \ell_1/n \leq \gamma} \frac{1}{s^n}\frac{n!}{\ell!} = s \sum_{\ell_1 = 0}^{\lfloor \gamma n \rfloor} \frac{1}{s^n}\frac{n!}{\ell_1!(n-\ell_1)!}\sum_{\ell_2 + \ldots + \ell_s \leq n - \ell_1} \frac{(n-\ell_1)!}{\ell_2!\cdot\ldots\cdot\ell_s!} \\
	&= s \sum_{\ell_1 = 0}^{\lfloor \gamma n \rfloor} \frac{1}{s^n}{n \choose n-\ell_1} (s-1)^{n-\ell_1} = s \sum_{k = n - \lfloor \gamma n \rfloor}^n {n \choose k}\left( 1 - \frac{1}{s} \right)^k\left( \frac{1}{s} \right)^{n-k}.
\end{align*}
Let $X_1, X_2, \ldots$ be i.i.d. Bernoulli random variables such that $\p{X_1 = 0} = 1/s = 1 - \p{X_1 = 1}$. Denote $S_n = X_1 + \ldots + X_n$. Then the last expression equals $s\p{S_n \geq n - \lfloor \gamma n \rfloor }$ and can be estimated from above as follows
\[
	s\p{S_n \geq n - \gamma n } = s\p{\frac{S_n - \E S_n}{n} \geq \frac{1}{s} - \gamma} \leq s\exp \big( -2n(1/s - \gamma)^2 \big) \xrightarrow[n\to\infty]{} 0,
\]
where the last inequality follows for instance from Hoeffding's inequality (see \cite{ZZ}).
\end{proof}

\begin{lm}\label{lm.density}
	Let $X$ be a random vector in $\R^n$ with a bounded density. Let $\fun{A}{\R^n}{\R^k}$ be a linear mapping of rank $r$. Then there exists a constant $C$ such that for any intervals $I_1, \ldots, I_k \subset \R$ of finite length we have
	\[
		\p{AX \in I_1 \times \ldots I_k}	\leq C |I_{i_1}|\cdot\ldots\cdot |I_{i_r}|,
	\]
	where $1 \leq i_1 < \ldots < i_r \leq k$ are indices of those rows of the matrix $A$ which are linearly independent. 
\end{lm}
\begin{proof}
Let $a_1, \ldots, a_k \in \R^n$ be rows of the matrix $A$. We know there are $r$ of them, say $a_1, \ldots, a_r$, which are linearly independent. Thus there exists an invertible $r \times r$ matrix $U$ such that
\[
	U\begin{bmatrix} a_1 \\ \vdots \\ a_r \end{bmatrix} = \begin{bmatrix} e_1 \\ \vdots \\ e_r \end{bmatrix} =: E,
\]
where $e_i \in \R^n$ is the $i$-th vector of the standard basis of $\R^n$. Notice that
\begin{align*}
	\p{AX \in I_1 \times \ldots \times I_k} &\leq \p{U^{-1}EX \in I_1 \times \ldots \times I_r} = \p{(X_1, \ldots, X_r) \in U(I_1 \times \ldots I_r)} \\
	&\leq C|U(I_1\times \ldots \times I_k)| = C|\det U|\cdot|I_1|\cdot\ldots\cdot |I_r|,
\end{align*}
for the vector $(X_1, \ldots, X_r)$ also has a bounded density on $\R^r$. This finishes the proof.
\end{proof}

\begin{lm}\label{lm.matrixrank}
	Let $A$ be a matrix of dimension $k\times j$, with entries in $\{0,1\}$, and satisfying the following 			conditions
	\begin{enumerate}[(i)]
		\item no two columns are equal.	 	
   	\item no two rows are equal.  
   	\item one column consists of all $1$s.  
	\end{enumerate}
	Then, the rank of $A$ is at least $\min(k, \lfloor \log_2 j\rfloor+1)$.
\end{lm}
\begin{proof}{\it (Due to Dima Gourevitch)}
  Denote $r = \textrm{rank} A$. The assertion of the lemma is equivalent to the statement that $2^r \geq j$ and if $2^r = j$ then $r = k$.

We may assume without loss of generality that the first $r$ rows of $A$ are linearly independent and the others are their linear combinations. Under this assumption, if two columns are identical in the first $r$
coordinates then they are identical in all coordinates. By condition \emph{(i)}, such columns do not exist. Therefore the $r \times j$ submatrix $B$ which consists of the first $r$ rows has distinct columns. As a result $j \leq 2^r$.

Now suppose $j = 2^r$. If $k > r$, consider the $r+1$ row of $A$. It is a linear combination of the first $r$ rows. Since the columns of $B$ include the column $e_i = (0,..,0,1,0,..,0)$ for all $i=1,\ldots,r$, the coefficient of each row is either $0$ or $1$. Since $A$ includes a column of $1$s, the coefficient of exactly one row is $1$, and all other coefficients vanish. Thus, the $r+1$-th row is identical to one of the first $r$ rows - in contradiction to condition \emph{(ii)}. 
\end{proof}

\begin{proof}[Proof of Theorem \ref{thm.cue^}]

Fix an integer $k \geq 1$ and finite intervals $I_1, \ldots, I_k \subset \R_+$ which are mutually disjoint.  We need
to compute the probability of the event $\{\tau_M(I_j) > 0, j = 1, \ldots, k\}$ which means that in each interval $I_j$ there is a rescaled eigenphase. Such eigenphase is of the form $\frac{2^M}{2\pi}\left( \theta^{\epsilon_1}_1 + \ldots + \theta^{\epsilon_M}_M \modulo 2\pi \right)$ for some $\epsilon = (\epsilon_1, \ldots, \epsilon_M) \in \{1, 2\}^M$. Therefore
\[
	\left\{\tau_M(I_j) > 0, j = 1, \ldots, k\right\} = \bigcup_\epsilon A_\epsilon,
\]
where
\begin{equation}\label{eq.defA_epsilon}
	A_\epsilon = \bigg\{ \sum_{i=1}^M \theta_i^{\epsilon^j_i} \modulo 2\pi \in \underbrace{\frac{2\pi}{2^M}I_j}_{J_j}, j = 1, \ldots, k \bigg\},
\end{equation}
and $\epsilon$ runs over the set 
\begin{equation}\label{eq.defE}
	\mathcal{E} = \left\{[\epsilon_i^j]_{i = 1,\ldots, M}^{j = 1, \ldots, k} \ | \ \epsilon_i^j \in \{1,2\}, \ \epsilon^u \neq \epsilon^v, \textrm{for $u \neq v, u, v = 1, \ldots, k$}\right\}
\end{equation}
of all $k \times M$ matrices with entries $1, 2$ which have pairwise distinct rows $\epsilon^j = (\epsilon^j_1, \ldots, \epsilon^j_M) \in \{1,2\}^M$, $j = 1, \ldots, k$ ($j$-th row $\epsilon^j$ describes the $j$-th eigenphase and since intervals are disjoint we assume the rows are distinct). Column vectors are denoted by $\epsilon_i = [\epsilon_i^1, \ldots, \epsilon_i^k]^T$, $i = 1, \ldots, M$.

We say that $\epsilon$ is \emph{bad} if the collection of its vector columns $\left\{ \epsilon_i, i = 1, \ldots, M \right\}$ is less than $2^k$. Otherwise $\epsilon$ is called \emph{good}.

Obviously,
\[
	\p{\bigcup_{\textrm{good $\epsilon$'s}} A_\epsilon} \leq \p{\bigcup_{\epsilon} A_\epsilon} \leq \p{\bigcup_{\textrm{good $\epsilon$'s}} A_\epsilon} + \p{\bigcup_{\textrm{bad $\epsilon$'s}} A_\epsilon}.
\]
The strategy is to show that the contribution of bad $\epsilon$'s vanishes for large $M$ while good $\epsilon$'s essentially provide the desired result $\prod_j |I_j|$ when $M$ goes to infinity. So the proof will be divided into several parts.

\paragraph*{Good $\epsilon$'s.}

The goal here is to prove
\begin{equation}\label{eq.goodgoal}
	\lim_{\max_j |I_j| \to 0}\lim_{M\to\infty} \frac{1}{|I_1|\cdot
	\ldots\cdot |I_k|}\p{ \bigcup_{\textrm{good $\epsilon$'s}} 
	A_\epsilon } = 1,
\end{equation}
with the required uniformity in the choice of the disjoint intervals $I_j$.
By virtue of
\[
	\sum_{\textrm{good $\epsilon$'s}} \p{A_\epsilon} - \sum_{ \substack{\textrm{good $\epsilon, \widetilde{\epsilon}$} \\ \epsilon \neq \widetilde{\epsilon}} } \p{A_\epsilon \cap A_{\widetilde{\epsilon}}} \leq \p{ \bigcup_{\textrm{good $\epsilon$'s}} A_\epsilon } \leq \sum_{\textrm{good $\epsilon$'s}} \p{A_\epsilon}
\]
it suffices to prove that
\begin{equation}\label{eq.thingone}
	\lim_{M\to\infty} \sum_{\textrm{good $\epsilon$'s}} \p{A_\epsilon} = \prod |I_j|
\end{equation}
uniformly,
and that the correlations between two different good epsilons does not matter
\begin{equation}\label{eq.thingtwo}
	\limsup_{\max_j |I_j| \to 0}\limsup_{M\to\infty} \frac{1}{\prod |I_j|}\sum_{ \substack{\textrm{good $\epsilon, \widetilde{\epsilon}$} \\ \epsilon \neq \widetilde{\epsilon}} } \p{A_\epsilon \cap A_{\widetilde{\epsilon}}} = 0.
\end{equation}

Let us now prove \eqref{eq.thingone}. The proof of \eqref{eq.thingtwo} is deferred to the very end as we will need the ideas developed here as well as in the part devoted to bad $\epsilon$'s.

Given $\epsilon \in \mc{E}$ and a vector $\alpha = [\alpha_1 \ldots \alpha_k]^T \in \{1,2\}^k$ we count how many column vectors of $\epsilon$ equals $\alpha$ and call this number $\ell_\alpha$. Then $\sum_\alpha \ell_\alpha = M$. Note that $\epsilon$ is good iff all $\ell_\alpha$s are nonzero. The crucial observation is that the probability of the event $A_\epsilon$ does depend only on the vector $\ell = (\ell_\alpha)_{\alpha \in \{1,2\}^k}$ associated to $\epsilon$ as described before. Indeed, the sum $\sum_{i=1}^M [\theta_i^{\epsilon_i^1} \ldots \theta_i^{\epsilon_i^k}]^T \modulo 2\pi$ is identically distributed as the random vector $\sum_\alpha \psi(\alpha, \ell_\alpha) \modulo 2\pi$, where
\begin{equation}\label{eq.psi}
	\psi (\alpha, \ell_\alpha) = \begin{bmatrix} \psi_1 (\alpha, \ell_\alpha) \\ \vdots \\ \psi_k (\alpha, \ell_\alpha) \end{bmatrix} = \begin{bmatrix} \theta_{i_1}^{\alpha_1} \\ \vdots \\ \theta_{i_1}^{\alpha_k} \end{bmatrix} + \ldots + \begin{bmatrix} \theta_{i_{\ell_\alpha}}^{\alpha_1} \\ \vdots \\ \theta_{i_{\ell_\alpha}}^{\alpha_k} \end{bmatrix} \modulo 2\pi
\end{equation}
is a sum modulo $2\pi$ of i.i.d. vectors. Note that the distribution of $\psi(\alpha, \ell_\alpha)$ does not depend on the choice of indices $i_1, \ldots, i_{\ell_\alpha}$ but only on $\alpha$ and $\ell_\alpha$. Consequently, denoting by $\mc{E}_\ell$ the set of all $\epsilon$'s such that there are exactly $\ell_\alpha$ indices $1 \leq i_1 < \ldots < i_{\ell_\alpha} \leq M$ for which $\epsilon_{i_1} = \ldots = \epsilon_{i_{\ell_\alpha}} = \alpha$, we have that the value of $\p{A_\epsilon}$ is the same for all $\epsilon \in \mc{E}_\ell$. Clearly $\sharp \mc{E}_\ell = \frac{M!}{\ell !}$, whence
\begin{equation}\label{eq.C_l}
	\sum_{\textrm{good $\epsilon$'s}} \p{A_\epsilon} = \sum_{\textrm{good $\ell$'s}} \frac{M!}{\ell !}\p{ \sum_{\alpha \in \{1,2\}^k} \psi(\alpha, \ell_\alpha) \modulo 2\pi \in J_1 \times \ldots \times J_k }.
\end{equation}

The idea is to identify those terms which will sum up to $\prod |I_i|$ and the rest which will be neglected in the limit of large $M$. To do this, set a positive parameter $\gamma < 1/2^k$ and let us call a good $\ell$ \emph{very good} (\emph{v.g.} for short) if $\ell_\alpha > \gamma M$ for every $\alpha$ and \emph{quite good} (\emph{q.g.} for short) otherwise. We claim that
\begin{equation}\label{eq.claim1}\tag{C1}
	\p{ \sum \psi(\alpha, \ell_\alpha) \modulo 2\pi \in J_1 \times \ldots \times J_k } \leq C \prod |J_j|, \qquad \textrm{\emph{for a good} $\ell$},
\end{equation}
and
\begin{equation}\label{eq.claim2}\tag{C2}
\begin{split}
	\p{ \sum \psi(\alpha, \ell_\alpha) \modulo 2\pi \in J_1 \times \ldots \times J_k } =  \frac{\prod |J_j|}{(2\pi)^k}\left( 1 + \frac{r_\ell}{\sqrt{M}} \right), \qquad |r_\ell| \leq C, \\
	\textrm{\emph{for a very good} $\ell$},
\end{split}
\end{equation}
\emph{where $C$ is a constant} (from now on in this proof we adopt the convention that $C$ is a constant depending only on $k$ which may differ from line to line). 

Let us postpone the proofs and see how to conclude \eqref{eq.thingone}. Notice that $\frac{\prod |J_j|}{(2\pi)^k} = \frac{1}{2^{kM}}\prod |I_j|$. Thus applying \eqref{eq.claim1} we obtain
\[
	\sum_{\textrm{q.g. $\ell$'s}}\p{ \sum \psi(\alpha, \ell_\alpha) \modulo 2\pi \in J_1 \times \ldots \times J_k} \leq \prod |I_j|\cdot C \sum_{\textrm{q.g. $\ell$'s}} \frac{1}{2^{kM}}\frac{M!}{\ell!}.
\]
By Lemma \ref{lm.q.g.sum} it vanishes when $M\to\infty$. Now we deal with very good $\ell$'s writing with the aid of \eqref{eq.claim2} that
\begin{align*}
	\sum_{\textrm{v.g. $\ell$'s}}\p{ \sum \psi(\alpha, \ell_\alpha) \modulo 2\pi \in J_1 \times \ldots \times J_k} = \prod |I_j| \Bigg(& \sum_{\textrm{v.g. $\ell$'s}} \frac{1}{2^{kM}}\frac{M!}{\ell!} \\
	&+ \sum_{\textrm{v.g. $\ell$'s}} \frac{1}{2^{kM}}\frac{M!}{\ell!}\frac{r_\ell}{\sqrt{M}}\Bigg).
\end{align*}
The first term in the bracket approaches $1$ in the limit $M\to\infty$ due to Lemma \ref{lm.q.g.sum}, while the second one approaches $0$ as it is bounded
above by $C\frac{1}{\sqrt{M}}$.

\begin{proof}[Proof of \eqref{eq.claim1}]
Let us define the vectors
\[
	e_j = (\underbrace{2, \ldots, 2}_{j-1}, 1, \underbrace{2, \ldots, 2}_{k-j}) \in \{1,2\}^k, \qquad j = 1, \ldots, k.
\]
Since $\ell$ is good, in particular we have that $\ell_{e_j} > 0$, so denoting the random vector $\psi (e_j, \ell_{e_j})$ by $\Psi^j$ we have $\sum_\alpha \psi (\alpha, \ell_\alpha ) = (\Psi^1 + \ldots + \Psi^k) + \sum_{\alpha \notin \{e_1, \ldots, e_k\}} \psi (\alpha, \ell_\alpha)$. By independence it is enough to show that the random vector $\Psi = \Psi^1 + \ldots + \Psi^k \modulo 2\pi$ has a bounded density on $[0,2\pi)^k$. Equation \eqref{eq.psi} yields that
\[
	\Psi^j = (\underbrace{Y_j, \ldots, Y_j}_{j-1}, X_j, \underbrace{Y_j, \ldots, Y_j}_{k-j}),
\]
where $(X_j, Y_j)$ are independent random vectors on $[0, 2\pi)^2$ with the same distributions as the vectors $(\theta_1^1 + \ldots + \theta_{\ell_{e_j}}^1 \modulo 2\pi, \theta_1^2 + \ldots + \theta_{\ell_{e_j}}^2 \modulo 2\pi)$ respectively. Clearly, the vector $(X_j, Y_j)$ has a bounded density on $[0,2\pi)^2$ because the vector $(\theta_1^1, \theta_1^2)$ has a bounded density. Therefore the vector $(X_1, Y_1, \ldots, X_k, Y_k)$ has a bounded density on $[0, 2\pi)^{2k}$. A certain linear transformation with determinant $1$ maps this vector to $(\Psi^1 + \ldots + \Psi^k, Y_1, \ldots, Y_k)$ which consequently also has a bounded density. One projects it to the first $k$ coordinates and then takes care of addition modulo $2\pi$ obtaining that $\Psi$ has a bounded density, which finishes the proof.
\end{proof}

\begin{proof}[Proof of \eqref{eq.claim2}]
Given a vector $\alpha \in \{1,2\}^k$ let $\Theta^\alpha$ denote the random vector in $[0,2\pi)^k$ identically distributed as the vector $(\theta^{\alpha_1}_1, \ldots, \theta^{\alpha_k}_1)$. Take its independent copies $\Theta_1^\alpha, \Theta_2^\alpha, \ldots$ such that the family $\{\Theta_1^\alpha, \Theta_2^\alpha, \ldots\}_{\alpha \in \{1,2\}^k}$ also consist of independent random vectors. Then $\E \Theta^\alpha = [\pi, \ldots, \pi]^T$, and
\begin{align*}
	p_{\ell, M} &= \p{\sum_\alpha \psi(\alpha, \ell_\alpha) \modulo 2\pi \in J_1 \times \ldots \times J_k} = \p{\sum_\alpha \sum_{l=1}^{\ell_\alpha} \Theta^\alpha_l \modulo 2\pi \in J_1 \times \ldots \times J_k} \\
	&= \sum_{i_1, \ldots, i_k = 0}^{M-1} \p{\sum_\alpha \sum_{l=1}^{\ell_\alpha} \Theta^\alpha_l \in (J_1+2\pi i_1) \times \ldots \times (J_k+2\pi i_k)} \\
	&
	\!\!\!\!\!\!\!
	\!\!\!\!\!\!\!
	\!\!\!\!\!\!\!
	= \sum_i \p{\sum_\alpha \sum_{l=1}^{\ell_\alpha} \frac{\Theta^\alpha_l - \E\Theta^\alpha_l}{\sqrt{M}} \in \frac{1}{\sqrt{M}}(J_1+2\pi (i_1-M/2)) \times \ldots \times \frac{1}{\sqrt{M}}(J_k+2\pi (i_k-M/2))}.
\end{align*}
To ease the notation we introduce new indices 
\[j = \left( i_1 - \frac{M}{2}, \ldots, i_k - \frac{M}{2} \right) \in \left\{ -\frac{M}{2}, -\frac{M}{2} +1, \ldots, \frac{M}{2} - 1 \right\}^k,\]
sets 
\[ K_{j,M} = \frac{1}{\sqrt{M}}(J_1+2\pi j_1) \times \ldots \times \frac{1}{\sqrt{M}}(J_k+2\pi j_k), \]
and the vector 
\[ S_M = \sum_\alpha \sum_{l=1}^{\ell_\alpha} \frac{\Theta^\alpha_l - \E\Theta^\alpha_l}{\sqrt{M}}.\]

Now we intend to use the local Central Limit Theorem of \cite{BR}. Indeed, due to independence such a theorem should hopefully yield that $S_M$ has a normal distribution for large $M$. To be more precise, let us consider the matrix $\Cov S_M = \sum_\alpha \frac{\ell_\alpha}{M} \Cov \Theta^\alpha$ and its eigenvalues. Since for any $x \in \R^k$
\[
	x^T(\Cov S_M)x 
	=\sum_\alpha \frac{\ell_\alpha}{M} x^T(\Cov \Theta^\alpha)x \leq \underbrace{\max_\alpha \| \Cov \Theta^\alpha \|^{1/2}}_{C} |x|^2,
\]
it is clear that the largest eigenvalues are uniformly (i.e. with respect to $M$) bounded by $C$, which depends only on $k$. To provide an uniform bound for the smallest eigenvalues let us observe that (recall that $e_i$ is the vector $(2, \ldots, 2, 1, 2, \ldots, 2)$)
\[
	x^T(\Cov S_M)x \geq \sum_{j=1}^k \frac{\ell_{e_j}}{M} x^T(\Cov \Theta^{e_j})x > \gamma x^T\Big( \sum_{j=1}^k \Cov \Theta^{e_j} \Big)x \geq \gamma\cdot \frac{\pi^2}{3} |x|^2,
\]
where the second inequality is because $\ell$ is very good. 

It
is a matter of direct computation to see the last inequality as for $k \geq 2$ we have $\sum_{j=1}^k \Cov \Theta^{e_j} = \big((k-2)\pi^2/3 - 2\big)[1 \ldots 1]^T[1 \ldots 1] + \textrm{diag} (2 + 2\pi^2/3, \ldots, 2 + 2\pi^2/3)$ and for $k = 1$ the sum equals $\pi^2/3$. Therefore, with the matrix $B_M$ given by
\[
	B_M^2 = (\Cov S_M)^{-1}
\]
it holds that
\[
	\frac{1}{C}|x| \leq |B_Mx| \leq C|x|.
\]
Therefore the assumptions of \cite[Corollary 19.4]{BR} are satisfied (for the family of independent random vectors $\{\Theta_1^\alpha, \Theta_2^\alpha, \ldots\}_{\alpha \in \{1,2\}^k}$), so the vector $B_MS_M$ possesses a density $q_M$ and
\[
	\sup_{x \in \R^k} \big(1 + |x|^{k+2}\big)\left( q_M(x) - \phi(x) - \frac{1}{\sqrt{M}}P_M(x)\phi(x) \right) = O(M^{-k/2}),
\]
where $\phi(x) = \frac{1}{\sqrt{2\pi}^k}e^{-|x|^2/2}$ is the density of the standard normal distribution in $\R^k$ and $P_M$ is a polynomial of degree $k-1$ whose coefficients depends on the cumulants of the vectors $B_M\Theta^\alpha$. We may put it differently, i.e.
\[
	q_M(x) = \phi(x) + \frac{1}{\sqrt{M}}\bigg( \underbrace{P_M(x)\phi(x) + \frac{f_M(x)}{1 + |x|^{k+2}}}_{h_M(x)} \bigg),
\]
for some functions $f_M$ uniformly bounded $\sup_M \sup_{x \in \R^k} |f_M(x)| = C < \infty$. Therefore, denoting $L_{j,M} = B_MK_{j,M}$,
\begin{equation}\label{eq4}
\begin{split}
	p_{\ell, M} &= \sum_j \p{S_M \in K_{j,M}} = \sum_j \p{B_MS_M \in B_MK_{j,M}} = \sum_j \int_{L_{j,M}} q_M \\
	&= \sum_j \int_{L_{j,n}} \phi + \frac{1}{\sqrt{M}}\sum_j \int_{L_{j,n}} h_M = a_M + \frac{1}{\sqrt{M}}b_M.
\end{split}
\end{equation}

Let us firstly deal with the error term $b_M$. Denoting
\[
	\kappa = \frac{|J_1|\cdot \ldots \cdot|J_k|}{(2\pi)^k}
\]
we are to show that
\begin{equation}\label{eq5}
	|b_M| \leq C\kappa.
\end{equation}
To do this we estimate the integrated function
\[
	|h_M(x)| \leq |P_M(x)|\phi(x) + \frac{C}{1 + |x|^{k+2}} =: h(x).
\]
Then $|b_M| \leq \sum_j \int_{L_{j,M}} h.$ Introduce \emph{full} boxes
\[
	F_{j,M} = B_M\left(\frac{1}{\sqrt{M}}([0, 2\pi) + 2\pi j_1) \times \ldots \times \frac{1}{\sqrt{M}}([0, 2\pi) + 2\pi j_k)\right)
\]
and observe that
\[
	\int_{L_{j,M}} h = \frac{|L_{j,M}|}{|F_{j,M}|}|F_{j,M}|\frac{1}{|L_{j,n}|}\int_{L_{j,M}} h \leq \kappa |F_{j,M}|\sup_{L_{j,M}} h \leq \kappa |F_{j,M}|\sup_{F_{j,M}} h.
\]
Since $\textrm{diam} F_{j,M} \leq C\frac{2\pi\sqrt{k}}{\sqrt{M}} \xrightarrow[M \longrightarrow \infty]{} 0$, the sets $F_{j,M}$ are pairwise disjoint and sum up to $B_M[-\pi\sqrt{M}, \pi\sqrt{M})^k$, we can infer that the sum $\sum_j |F_{j,M}|\sup_{F_{j,M}} h$ converges to $\int_{\R^k} h = C < \infty$. Hence, this sum is bounded by $C$ and we get \eqref{eq5}.

Now we handle the main term $a_M$. We prove it equals $\kappa$ up to another error $\kappa\frac{C}{\sqrt{M}}$. Let $A_{j,M}: \R^k \longrightarrow \R^k$ be the linear isomorphism mapping $F_{j,M}$ onto $L_{j,M}$. It equals $B_M\tilde{A}_{j,M}B_M^{-1}$, where $\tilde{A}_{j,M}$ is the linear mapping transforming the box $B_M^{-1}F_{j,M}$ onto the box $B_M^{-1}L_{j,M}$, whence $|\det A_{j,M}| = \kappa$. Thus, changing the variable we obtain
\[
	\int_{L_{j,M}} \phi(x) \dd x = \kappa \int_{F_{j,M}} \phi(A_{j,M}x) \dd x.
\]
Notice that $A_{j,M}x$ is close to $x$, whenever $x \in F_{j,M}$, for
\[
	|A_{j,n}x - x| \leq \textrm{diam} F_{j,M}, \qquad x \in F_{j,M}.
\]
Consequently, on $F_{j,M}$, $\phi(A_{j,M}x)$ is close to $\phi(x)$. Strictly, we use the mean value theorem and get
\[
	\int_{L_{j,M}} \phi(x) \dd x = \kappa \int_{F_{j,M}} \phi(x) \dd x + \kappa \int_{F_{j,M}} \nabla \phi_V(\eta_x)\cdot (A_{j,M}x - x) \dd x,
\]
for some mean points $\eta_x \in [x, A_{j,M}x]$. This results in
\begin{align*}
	a_M &= \sum_j \int_{L_{j,M}} \phi(x) \dd x = \kappa \sum_j \int_{F_{j,M}} \phi + \kappa \int_{\bigcup_j F_{j,M}} \nabla \phi(\eta_x)\cdot (A_{j,M}x - x) \dd x \\
	&= \kappa \bigg( 1 \underbrace{- \int_{\R^k \setminus B_M[-\pi\sqrt{M}, \pi\sqrt{M})^k} \phi}_{c_M} + \underbrace{\sum_j \int_{F_{j,M}} \nabla \phi_V(\eta_x)\cdot (A_{j,M}x - x) \dd x}_{d_M} \bigg).
\end{align*}
We are almost done. Clearly $c_M$ converges to $0$ faster that $1/\sqrt{M}$, so $|c_M| \leq C/\sqrt{M}$. For $d_M$ we use the Schwarz inequality and integrability of $|\nabla \phi(\eta_x)|$
\[
	|d_M| \leq \sum_j \int_{F_{j,M}} |\nabla \phi(\eta_x)||A_{j,M}x - x| \dd x \leq \textrm{diam} F_{j,M} \int_{\bigcup F_{j,M}} |\nabla \phi(\eta_x)| \dd x \leq \frac{C}{\sqrt{M}}.
\]
This completes the proof of \eqref{eq.claim2}.
\end{proof}

We have proved claims \eqref{eq.claim1} and \eqref{eq.claim2},
so the proof of the part concerning good $\epsilon$'s is now complete. Let us proceed to tackle bad $\epsilon$'s.

\paragraph*{Bad $\epsilon$'s.}

The goal here is to show that
\begin{equation}\label{eq.badgoal}
\lim_{M\to\infty} \p{ \bigcup_{\textrm{bad $\epsilon$'s}} A_\epsilon } = 0,
\end{equation}
again, with the required uniformity.
Obviously it suffices to show that $\sum_{\textrm{bad $\epsilon$'s}} \p{A_\epsilon} \xrightarrow[M\to\infty] {} 0$. Let $\mc{F}_j$ be the set of those bad $\epsilon$'s for which the cardinality of the set $\left\{ \epsilon_i, i = 1, \ldots, M \right\}$ equals $j$. Observe that $\sharp \mc{F}_j \leq j^M$. With the aid of Lemma \ref{lm.matrixrank} we will show that
\begin{equation}\label{eq.badepsilonfinal}
	\forall \epsilon \in \mc{F}_j \ \p{A_\epsilon} \leq C \cdot 2^{-M(1 + \lfloor \log_2 j \rfloor)}\cdot O\big((\max_j |I_j|)^{1 + \lfloor \log_2 j \rfloor}\big), \quad \textrm{when $\max_j |I_j| \longrightarrow 0$}.
\end{equation}
This will finish the proof, for
\begin{equation}\label{eq.argumentofbadepsilons}
	\begin{split}
		\sum_{\textrm{bad $\epsilon$'s}} \p{A_\epsilon} &\leq C\cdot O(\max_j |I_j|)\sum_{j=1}^{2^k-1} j^M\cdot 2^{-M(1 + \lfloor \log_2 j \rfloor)} \\
	&= C\cdot O(\max_j |I_j|)\sum_{j=1}^{2^k-1} 2^{-M(1 + \lfloor \log_2 j \rfloor - \log_2 j)} \xrightarrow[M\to\infty]{} 0.
	\end{split}
\end{equation}

For the proof of \eqref{eq.badepsilonfinal} fix $\epsilon \in \mc{F}_j$. We have seen that
\[
	\p{A_\epsilon} = \p{ \sum \psi(\alpha, \ell_\alpha) \modulo 2\pi \in J_1 \times \ldots \times J_k }
\]
and we know that there are exactly $j$ numbers $\ell_\alpha$ which are nonzero, say those which correspond to vectors $\alpha^1, \ldots, \alpha^j \in \{1,2\}^k$. Denote $\Psi^j = \psi(\alpha^i, \ell_{\alpha^i})$, $i = 1, \ldots, j$ and consider the random vector $S_j = \Psi^1 + \ldots + \Psi^j$ in $\R^k$. As in the proof of Claim \eqref{eq.claim1} we observe that $S_j$ is a linear image of the vector $(X_1, Y_1, \ldots, X_j, Y_j)$. This mapping is given by the matrix $A = [a_{uv}]$ where
\[
	a_{2i-1, v} = \begin{cases} 1, & \alpha^i_v = 1 \\ 0, & \alpha^i_v = 2 \end{cases}, \qquad a_{2i, v} = \begin{cases} 0, & \alpha^i_v = 1 \\ 1, & \alpha^i_v = 2 \end{cases}.
\]
By Lemma \ref{lm.density} we obtain
\begin{equation}\label{eq.forcor2}
	\p{S_j \modulo 2\pi \in J_1\times \ldots \times J_k} \leq C\max\big(|J_{i_1}|\cdot\ldots\cdot |J_{i_r}|\big) = C\cdot O(\max_j |I_j|)\cdot 2^{-Mr},
\end{equation}
where $r = \textrm{rank} A$. The number $r$ does not change if we replace the $2i$-th column of $A$ with the vector $e$ with $1$ at each its entry, as the sum of $2i-1$-th and $2i$-th columns is just $e$. Taking only the columns $1,2,3, 5, \ldots, 2j-1$ we get the matrix $B$ which has the same rank as $A$. It has $j+1$ columns and fulfils the assumptions of Lemma \ref{lm.matrixrank}. Thus $r \geq \min (1 + \lfloor \log_2 (1 + j) \rfloor, k)$ and when $j < 2^k - 1$ this minimum equals $1 + \lfloor \log_2 (1 + j) \rfloor \geq 1 + \lfloor \log_2 j \rfloor$. If $j = 2^k - 1$ in the matrix $A$ there must be two identical columns, one with even, say $2u$, and one with odd, say $2v-1$ index, which means that the $u$-th and the $v$-th column of $B$ add up to $e$, so the $v$-th column may be erased and the rank of $B$ does not change. Therefore we apply the lemma to the matrix $B$ with erased the $v$-th column which is of size $k\times j$ and get again $r \geq \min (1 + \lfloor \log_2
  j \rfloor, k) = 1 + \lfloor \log_2 j \rfloor$. This completes the proof of \eqref{eq.badepsilonfinal}.

\paragraph*{Pairs of good $\epsilon$'s, i.e. the proof of \eqref{eq.thingtwo}.}

We denote by $\Theta_i(\epsilon)$ the random vector $(\theta_i^{\epsilon_i^1}, \ldots, \theta_i^{\epsilon_i^k})$. By the definition of $A_\epsilon$ we may write
\begin{equation}\label{eq.forcor2'}
	A_\epsilon \cap A_{\widetilde{\epsilon}} = \left\{ \sum_{i=1}^M \begin{bmatrix} \Theta_i(\epsilon) \\ \Theta_i(\widetilde{\epsilon}) \end{bmatrix}	\modulo 2\pi \in \begin{array}{r} J_1 \times \ldots \times J_k \\ J_1 \times \ldots \times J_k \end{array} \right\}.
\end{equation}
Since the intervals $J_u$ and $J_v$ are disjoint for $u \neq v$, we may restrict ourselves to those $\epsilon$ and $\widetilde{\epsilon}$ for which $\epsilon^u \neq \widetilde{\epsilon}^v$ whenever $u \neq v$, $u, v = 1, \ldots, k$ as otherwise the event $A_\epsilon \cap A_{\widetilde{\epsilon}}$ is impossible. However it might happen that $\epsilon^u = \widetilde{\epsilon}^u$. Let us count for how many $u$'s it takes place, i.e. given $s \in \{1, \ldots, k\}$ let $\mc{P}_s$ be the set of all considered unordered pairs $\{\epsilon, \widetilde{\epsilon}\}$ for which there are exactly $k-s$ indices $1 \leq u_1 < \ldots < u_{k-s} \leq k$ such that $\epsilon^{u_j} = \widetilde{\epsilon}^{u_j}$, $j = 1, \ldots, k-s$. The value $s = 0$ is excluded as $\epsilon \neq \widetilde{\epsilon}$. We have
\[
	\sum_{\epsilon \neq \widetilde{\epsilon}} \p{A_\epsilon \cap A_{\widetilde{\epsilon}}} = \sum_{s = 1}^k \sum_{\{\epsilon, \widetilde{\epsilon}\} \in \mc{P}_s} \p{A_\epsilon \cap A_{\widetilde{\epsilon}}}.
\]
Thus we fix $s$ and prove that $\limsup_{\max_j |I_j| \to 0}\limsup_{M\to\infty} \frac{1}{\prod |I_j|}\sum_{\{\epsilon, \widetilde{\epsilon}\} \in \mc{P}_s} \p{A_\epsilon \cap A_{\widetilde{\epsilon}}} = 0$. There are two cases. A pair $\{\epsilon, \widetilde{\epsilon}\} \in \mc{P}_s$ can be \emph{good} which means $\sharp \left\{ \left[ \begin{smallmatrix} \epsilon_i \\ \widetilde{\epsilon}_i \end{smallmatrix} \right], \ i = 1, \ldots, M \right\} \geq 2^{k+s}$, or, otherwise we call it \emph{bad}. We obtain a decomposition $\mc{P}_s = \mc{P}_s^{\textrm{good}} \cup \mc{P}_s^{\textrm{bad}}$. Now for a good pair, applying the reasoning already used for bad $\epsilon$'s, i.e. combining lemmas \ref{lm.density} and \ref{lm.matrixrank}, we get the estimate
\[
	\p{A_\epsilon \cap A_{\widetilde{\epsilon}}} \leq C|J_1|\cdot \ldots \cdot |J_k| \big( \max_{j=1,\ldots, k} |J_j| \big)^s = \frac{C}{2^{(k+s)M}}\bigg(\prod |I_j|\bigg) \big( \max_j |I_j| \big)^s.
\]
But $\sharp \mc{P}_s^{\textrm{good}} \leq \sharp \mc{P}_s \leq {k \choose s}\cdot 2^{(k+s)M}$, so
\[
	\limsup_{\max_j |I_j| \to 0}\limsup_{M\to\infty} \frac{1}{\prod |I_j|}\sum_{\{\epsilon, \widetilde{\epsilon}\} \in \mc{P}_s^{\textrm{good}}} \p{A_\epsilon \cap A_{\widetilde{\epsilon}}} = 0.
\]

For a bad pair $\{\epsilon, \widetilde{\epsilon}\}$ we know that there are $k + s$ different rows and at most $2^{k+s} - 1$ different columns in the matrix $\left[ \begin{smallmatrix} \epsilon \\ \widetilde{\epsilon} \end{smallmatrix} \right]$. Hence we repeat the argument of the part concerning bad $\epsilon$'s. Namely, first exactly in the same manner as in that part we use Lemma \ref{lm.matrixrank} in order to establish an appropriate inequality in the spirit of \eqref{eq.badepsilonfinal}. Then we follow the estimate of \eqref{eq.argumentofbadepsilons} and conclude that
\[
	\lim_{M\to\infty} \sum_{\{\epsilon, \widetilde{\epsilon}\} \in \mc{P}_s^{\textrm{bad}}} \p{A_\epsilon \cap A_{\widetilde{\epsilon}}} = 0.
\]
This finishes the proof of Theorem
\ref{thm.cue^}.
\end{proof}

\begin{proof}[Proof of Corollary \ref{cor.cue^N}]
 Fix $\Delta$ small so that $s/\Delta$
 is an integer and
 divide the interval $[0,s]$ into  consecutive intervals of length
 $\Delta$, denoted $I_i$. Let
 $Z_i=\tau_M(I_i)$ and $\bar Z_i={\bf 1}_{\{Z_i>0\}}$. 
 Of course, $\tau_M([0,s])=\sum_{i=1}^{s/\Delta} Z_i$.
 Our goal is to show that $\tau_M([0,s])$ becomes Poissonian in
the limit of large $M$, from which the statement of the corollary
follows immediately.

 The proof of Theorem \ref{thm.cue^} yields the following facts.
 There exist a sequence 
 %$\epsilon_M$, 
 $\delta_{M,\Delta,k}$ with
 $$\limsup_{\Delta\to 0}\limsup_{M\to\infty} \delta_{M,\Delta,k} =0\,,$$
 and a universal constant $C$
 such that the following hold.
 \begin{eqnarray}
   \p{Z_i\neq \bar Z_i}&\leq & C \Delta^2\,, \label{eq:030312a}\\
   \E\left(\prod_{i\in J_k}
   \bar Z_i\right)&=& \Delta^k(1+O(\delta_{M,\Delta,k}))
   %+O(\epsilon_M) \Delta\,,
   \label{eq:030312b}
 \end{eqnarray}
 where $J_k$ denotes an arbitrary subset of $k$ distinct integers in
   $\{1,\ldots,s/\Delta\}$. 
   
   Indeed, to justify \eqref{eq:030312a} notice that
   \[
   \p{Z_i \neq \bar Z_i} = \p{\tau_M(I_i) \geq 2} \leq \p{\bigcup_{\epsilon \neq \widetilde{\epsilon}} A_\epsilon \cap A_{\widetilde{\epsilon}} },
   \]
   where $\epsilon, \widetilde{\epsilon} \in \{1,2\}^M$ and $A_\epsilon$ is the event that there is an eigenphase described by $\epsilon$ in the interval $I_i$ (see \eqref{eq.defA_epsilon}). The probability of the event $A_\epsilon \cap A_{\widetilde{\epsilon}}$ can be estimated by $C\cdot 2^{-2M}\cdot |I_i|^2 = 2^{-2M}\cdot C\Delta^2$. To see this, recall \eqref{eq.forcor2'} and follow the same argument which led to estimate \eqref{eq.forcor2} (in this case the relevant matrix has the rank no less than $2$). It suffices, as $\p{\bigcup_{\epsilon \neq \widetilde{\epsilon}} A_\epsilon \cap A_{\widetilde{\epsilon}} } \leq {2^M \choose 2}\cdot 2^{-2M}\cdot C\Delta^2 \leq C\Delta^2$. For \eqref{eq:030312b}, observe that
   \[
	   \E \prod_{i \in J_k} \bar Z_i = \E \1_{\{Z_i > 0, i \in J_k\}} = \p{ \tau_M(I_i) > 0, i \in J_k},
   \]
   and apply Theorem \ref{thm.cue^} (with its uniformity statement).
   %Thus, take
   %\[
   %\delta_{M,\Delta,k} = \max_{J_k \subset \{1,\ldots, s/\Delta\}} \left( \frac{1}{\Delta^k}\p{ \tau_M(I_i) > 0, i \in J_k} - 1 \right).
   %\]
   %Then $\E \prod_{i \in J_k} \bar Z_i \leq \Delta^k(1 + \delta_{M,\Delta,k})$ and this is exactly the assertion of Theorem \ref{thm.cue^} which guarantees that $\lim_{\Delta\to 0}\lim_{M\to\infty} \delta_{M,\Delta,k} =0$.
   
   Let $Y_i$ be i.i.d. Bernoulli random variables with
   $\p{Y_1=1}=1-\p{Y_1=0}=\Delta$. By \eqref{eq:030312b} we have that 
   for any integer $\ell$,
   $$
   \limsup_{\Delta\to 0}
   \limsup_{M\to\infty}\left| \E\left(\sum_{i=1}^{s/\Delta} \bar Z_i\right)^\ell 
   - \E\left(\sum_{i=1}^{s/\Delta} Y_i\right)^\ell\right| 
   =0\,.$$
   Since $\sum_{i=1}^{s/\Delta} Y_i$ converges to a Poisson
random variable of parameter $s$ as $\Delta\to 0$, it follows that
   $\sum_{i=1}^{s/\Delta} \bar Z_i$
   converges in distribution to a Poisson variable of parameter $s$, 
   when first $M\to\infty$ and then $\Delta\to 0$. 
   On the other hand,
   using \eqref{eq:030312a} we have that
   $$\p{\sum_{i=1}^{s/\Delta} \bar Z_i\neq 
   \sum_{i=1}^{s/\Delta}  Z_i} \leq Cs\Delta\,,$$
   and therefore, one concludes that also 
   $\sum_{i=1}^{s/\Delta}  Z_i$
   converges in distribution to a Poisson variable of parameter $s$, when first
   $M\to\infty$ and then $\Delta\to 0$.
This yields the corollary.
\end{proof}

\section*{Acknowledgements.}

TT was partially supported by NCN Grant no. 2011/01/N/ST1/05960.
MS, MK and KZ were supported by
the SFB Transregio-12 project der Deutschen Forschungsgemeinschaft
and a grant financed by the Polish National Science Centre under the
contract number  DEC-2011/01/M/ST2/00379.
OZ was supported by 
NSF grant DMS-0804133 and by a grant from the Israel Science Foundation.

Part of the work was done while the first named author was 
participating in The Kupcinet-Getz International Summer 
Science School at the Weizmann Institute of Science in Rehovot, Israel.
We are grateful to the WIS for financial support making this possible.

We thank S. Jain and A. Pandey for comments related to 
the discussion in
\cite{Jain}.
Finally, we thank Dima Gourevitch for both providing and
allowing us to use his proof of
Lemma \ref{lm.matrixrank}.

\end{document}